\documentclass{amsart}
\usepackage[utf8]{inputenc}
\usepackage{palatino}
\usepackage{hyperref}
\usepackage{parskip}

\usepackage{amsmath,amsthm,amsfonts,color,graphicx,mathrsfs}
\theoremstyle{plain}
\newtheorem{thm}{Theorem}

\newtheorem{lem}[thm]{Lemma}
\newtheorem{claim}{Claim}

\theoremstyle{definition}

\theoremstyle{remark}
\newtheorem{remark}[thm]{Remark}

\newtheorem{question}{Question}

\title[Generalized Alexander's Theorem]{Piecewise Linear Generalized Alexander's \\ Theorem in Dimension at most 5}
\author{Sudipta Kolay}
\address{School of Mathematics \\ Georgia Institute of Technology\\Atlanta, GA 30332, USA}
\email{skolay3@math.gatech.edu}
\date{}

\begin{document}
\begin{abstract} We study piecewise linear co-dimension two embeddings of closed oriented manifolds in Euclidean space, and show that any such embedding can
always be isotoped to be a closed braid as long as the ambient dimension is at most five, extending results of Alexander (in ambient dimension three), and
Viro and independently Kamada (in ambient dimension four). We also show an analogous result for higher co-dimension embeddings.
\end{abstract}

\maketitle


\section{Introduction}
 A classical theorem of Alexander \cite{A} says that every oriented link in $\mathbb{R}^3$  is isotopic to a closed braid. 
This theorem has been used to study knot theory, for example the Jones Polynomial \cite{J} is a knot invariant defined using braids.
In this paper we study generalizations of Alexander's theorem in higher ambient dimension.

Braided surfaces were first introduced by Rudolph \cite{R} for surfaces with boundary, but the notion we will be using is due to Viro. Viro defined the 
notion of closed braid for closed oriented surface in $\mathbb{R}^4$, which can be thought of as closure of certain (the ones with trivial boundary) braided 
surfaces in the sense of Rudolph. Hilden, Lozano and Montesinos \cite{HLM}, using different terminology,
first studied braided embeddings to prove that each three manifold has a braided embedding in $\mathbb{R}^5$. The notion of braided embedding was defined in general by Etnyre and Furukawa \cite{EF}, and they have been studied previously 
by Carter and Kamada \cite{CK}.
 
 The first analogue of Alexander's theorem for surfaces is due to Rudolph \cite{R}, who showed that every oriented ribbon surface is smoothly isotopic to
 a closed braid. Alexander's theorem was generalized to closed oriented surfaces in $\mathbb{R}^4$ by Viro and independently by Kamada. 
Viro announced his results in a lecture in 1990, but his proof was never published. Kamada gave an alternative proof \cite{K1,K2} using the motion picture method
 to describe surfaces in $\mathbb{R}^4$.

The main result of this article is to show that in the piecewise linear category, Alexander's theorem can be generalized to ambient dimension 5.
\begin{thm}[\textbf{P.L. Generalized Alexander's Theorem}]\label{1}
 Any closed oriented piecewise linear $(n-2)$-link in $\mathbb{R}^n$ can be piecewise linearly isotoped to be a closed braid for $3\leq n \leq 5$.
\end{thm}

Our approach is similar to Alexander's original proof in \cite{A} (see also \cite[Theorem~2.1]{B}, or \cite[Theorem~4.2]{K2}) in the
classical case. We give an alternate proof of Kamada's generalization of Alexander's Theorem in dimension 4. For completeness, we also
include the proof of the classical case of dimension 3. We also recover another classical result of Alexander \cite{A0}, that says
that any closed oriented piecewise linear $k$-manifold is a piecewise linear branched cover over the sphere $S^k$, see Remark~\ref{r5}.

If Theorem~\ref{1} can be upgraded to the smooth category, then there are applications to contact geometry. Etnyre and Furukawa (see \cite[Theorem~1.27]{EF}) showed  that if Alexander's Theorem holds
in the smooth category (with the branch locus being a submanifold) in ambient dimension five, then any embedding of a closed oriented 3-manifold in $S^5$ can be isotoped to be a transverse contact embedding.

 One may wonder if there are some analogues of Alexander's theorem for higher co-dimension link. More precisely,

\begin{question}\label{q1}
 Given an natural number $k$ is there an natural number $n\geq k+2$ so that any closed oriented  $k$-manifold embeds in $\mathbb{R}^{n}$, and moreover any embedding is isotopic
 to a closed braid?
\end{question}

It is well known that the embedding problem holds as long as $n\geq 2k$, see \cite[Theorem~5.5]{RS} for piecewise linear category, and \cite{W} for smooth category. By Theorem~\ref{2} below, in the piecewise linear
 category we have for $k\geq 2$ and $n\geq 2k$, any embedding is isotopic to a closed braid, so the answer to Question~\ref{q1} is affirmative.
Moreover, we can ask given any  $k$-link in $\mathbb{R}^n$, is it always isotopic to a closed braid? The following result gives a partial answer
to that question.

\begin{thm}\label{2}
 Any closed oriented piecewise linear $k$-link in $\mathbb{R}^n$ can be piecewise linearly isotoped to be a closed braid for $2n\geq 3k+2$.
\end{thm}

\textit{Organization}. The paper is organized as follows: in the Section 2 we define closed braids and positive links, and we show that the notions of closed braid and positive link 
are equivalent, thereby reducing the braiding problem to isotoping a link to be positive. In the Section 3, we describe cellular moves, which will be used to replace a negative
simplex with some positive simplices. In Sections 4 and 5, we study co-dimension two and higher co-dimension embeddings repectively.
We will show that under the given hypotheses of Theorems~\ref{1} and ~\ref{2}, any closed link can be isotoped to be positive, completing the proofs.
At the end of either section, we ask some questions about other possible generalizations of Alexander's Theorem.

\textit{Acknowledgements}. The author is grateful to John Etnyre for introducing the problem and many useful discussions. The author would like
to thank James Conway for making helpful comments on earlier drafts of this paper.

\section{Closed Braids and Positive Links}

 We assume that all spaces are piecewise linear, all embeddings are piecewise linear and locally flat, all isotopies are piecewise linear and ambient,
 and all other maps (radial projections, coverings and branched coverings) are topological\footnote{Radial projections need not be piecewise linear, see Chapter 1 in \cite{RS}.}.
 By linear we will mean linear in the affine sense.
 
 Let $k$ and $l$ be natural numbers with $l\geq 2$. Let $f:M^k\rightarrow\mathbb{R}^{k+l}$ be an embedding of a closed oriented $k$-manifold (possibly disconnected), and we call the image a (co-dimension $l$) $k$-link. We will be mostly concerned with co-dimension two embeddings, i.e.\@ $l=2$.
 
 We say that $f$ is a \textit{co-dimension $l$ braided embedding} if $f(M)$ is contained in a regular neighborhood $N(S^k)=S^k\times D^l$ of the standard sphere
(unit sphere in $\mathbb{R}^{k+1}\subset\mathbb{R}^{k+l}$) such that the embedding composed with the projection to the sphere, $pr_1\circ f:M\rightarrow S^k$ is an oriented branched 
covering map. Note that in case $k=1$, we have $pr_1\circ f$ is just an oriented covering map since the branch locus is empty, also if further $l=2$, then $f(M)$ is 
a closed braid (in the classical sense).  We generalize this notion and call the image $f(M)$ of a co-dimension $l$ braided embedding $f$ to be a \textit{co-dimension $l$ closed braid}. We will
just say $f$ is a braided embedding and $f(M)$ is a closed braid if co-dimension is clear from the context. By \textit{braiding} we will mean isotoping a link 
to be a closed braid. We will identify $M$ with $f(M)$, and think of $f$ as an inclusion. A simplex of $M$ is understood to be in $\mathbb{R}^{k+l}$.

Let us choose (and fix) a $l-1$ dimensional subspace $\ell$ of $\mathbb{R}^{k+l}$, which will play the role of the braiding axis. Let $\pi:\mathbb{R}^{k+l}\rightarrow\mathbb{R}^{k+1}$ 
denote orthogonal projection to $\ell^\bot$, and let $O$ denote the origin of $\mathbb{R}^{k+1}$.

 We say that a $k$-simplex $\sigma=[p_0,...,p_k]$ in $\mathbb{R}^{k+l}$ is in \textit{general position} with respect to $\ell$ if
any of the following equivalent conditions hold:
\begin{enumerate}
 \item There is no hyperplane in $\mathbb{R}^{k+l}$ which contains both $\sigma$ and $\ell$.
 \item There is no hyperplane in $\mathbb{R}^{k+1}$ which contains both $\pi(\sigma)$ and $O$.
 \item The vectors $\pi(p_0),...,\pi(p_k)$ are linearly independent.
 \item The determinant of $[\pi(p_0)|\pi(p_1)|...|\pi(p_k)]$ is nonzero.
\end{enumerate}
 We can always assume each simplex is in general position (with respect to $\ell$), because if not, then by slightly perturbing the vertices, we can put it in general position.
 
\textit{General Position}. We will be needing several general position arguments, and we will outline the proof of one of them.
They all follow the same pattern: the degenerate case happens if and only if a continuous function vanishes. 
So, if the system was non-degenerate, then any slight perturbation does not change that fact, and if the system was degenerate,
it would be possible to make it non-degenerate with a slight perturbation.

 We say that a simplex $[p_0,...,p_k]$ in $\mathbb{R}^{k+l}$ in general position (with respect to $\ell$) is \textit{positive} if the simplex $[O,\pi(p_0),...,\pi(p_k)]$ has the 
standard orientation of $\mathbb{R}^{k+1}$ (i.e.\@ $[\pi(p_0)|\pi(p_1)|...|\pi(p_k)]$ has positive determinant), otherwise we say it is \textit{negative}. We say that
 an embedded link $f:M^k\rightarrow\mathbb{R}^{k+l}$ is a \textit{positive} (with respect to $\ell$) if the image of each simplex is in general position with respect to $\ell$
and positive. Hereafter the axis $\ell$ will be in the background, it will be understood that a simplex is positive/negative/in general position means it is positive/negative/in general position
with respect to $\ell$.

Let $p:\mathbb{R}^{k+1}\setminus O\rightarrow S^k$ be the radial projection. For any piecewise linear manifold $M^k$ with a given cellular decomposition, let $\delta M$ denote the union of all $(k-2)$-faces of cells of $M$.

The following theorem shows that to prove Theorem~\ref{1}, it suffices to show we can isotope any link to be positive.

\begin{thm}\label{3}
 Let $f:M^k\rightarrow\mathbb{R}^{k+l}\setminus \ell$ be an embedding, then the composition $h$ defined by
$$M^k\xrightarrow{f}\mathbb{R}^{k+l}\setminus \ell\xrightarrow{\pi}\mathbb{R}^{k+1}\setminus O\xrightarrow{p} S^k$$ is an oriented
branched covering map if and only if all simplices of $M$ are positive. In other words, the notions of closed braid and positive link are equivalent.
\end{thm}

\begin{proof} If $M$ is a closed braid, then the restriction of $h$ to any particular simplex $\sigma$ must be orientation preserving, and it follows
that all simplices of $M$ must be positive. 

Let us now assume that $M$ is a positive link. Let $\Sigma:=h(\delta M)$. 
 We will show that $h$ restricts to a covering map on $M\setminus\ h^{-1}(\Sigma)$. Now any point $x$ of 
$M\setminus\ h^{-1}(\Sigma)$ could either be an interior point of a $k$-simplex, or on the interior of a $(k-1)$-face shared by two $k$-simplices.
We will show that in both these cases, we can find a compact neighbourhood $N$ of $x$ such that $h|_N$ is injective.

  Let $x$ be in the interior of the $k$-simplex $\sigma=[p_0,...,p_k]$. Then for any $y$ in $\sigma$ we see that 
 the ray passing through $O$ and $\pi(y)$ meets $\pi(\sigma)$ exactly once, since $\pi(p_0),...,\pi(p_k)$ form a basis for $\mathbb{R}^{k+1}$.
 Thus in this case $h|_\sigma$ is injective.
 
  Let us now suppose that $x$ is in the interior of the intersection of the adjacent simplices $\sigma=[p_0,...,p_k]$ and $\tau=[p_1,q_0,p_2,...,p_k]$. 
 For $\sigma$ and $\tau$ to be compatible, the induced orientation on the $(k-1)$-face $\nu=[p_1,...,p_k]$ they share must be opposite, i.e.\@
 the determinant of the matrix $[\pi(p_0)|\pi(p_1)|...|\pi(p_k)]$ is positive, and the determinant of the matrix 
 $[\pi(q_0)|\pi(p_1)|...|\pi(p_k)]$ is negative. Suppose $y$ is in $\sigma$ and the ray passing through $O$ and $\pi(y)$ meets 
 $\pi(\tau\setminus\nu)$, then we see that we see that for some non-negative scalars $c_0,...,c_k, d_1,...,d_k$ and positive scalars $d_0,\lambda$ we have
 $$ c_0\pi(p_0)+c_1\pi(p_1)+...+c_k\pi(q_k)=\lambda(d_0\pi(q_0)+d_1\pi(p_1)+...+d_k\pi(q_k))$$ and so we have
 $c_0\pi(p_0)-\lambda d_0\pi(q_0)\in\text{ Span}\{\pi(p_1),...,\pi(p_k)\}$, and hence
    $$c_0\det[\pi(p_0)|\pi(p_1)|...|\pi(p_k)]=\lambda d_0\det[\pi(q_0)|\pi(p_1)|...|\pi(p_k)]$$
  which is a contradiction to our assumption that both $\sigma$ and $\tau$ are positive. Thus in this case $h|_{\sigma\cup\tau}$ is injective.
 
  Thus in either case, for a compact neighborhood $N$ of $x$, $h|_N$ is a continuous bijection between compact Hausdorff spaces and hence a homeomorphism onto its image.
  Thus $h|_{M\setminus h^{-1}(\Sigma)}:M\setminus h^{-1}(\Sigma)\rightarrow S^k\setminus\Sigma$
  is a local homeomorphism, and infact a covering map since for any $y\in S^k\setminus\Sigma$, the fiber $h^{-1}(y)$ is compact and discrete.
  Also we can check that $h$ is orientation preserving. Thus $h$ is an oriented branched covering, as required.  
\end{proof}

\begin{remark}\label{r1} The map $h$ above is only continuous since the radial projection $p$ is so. However, we can compose $\pi\circ f$ with
the pseudo-radial projection\footnote{Pseudo-radial projection is the linear extension of the restriction of the radial projection to the vertices of the
domain, see Chapter 2 in \cite{RS}.} instead of the radial projection $p$, and then the resulting composition will be a piecewise linear branched cover.
\end{remark}

Let us choose (and fix) a unit vector $v\in\ell\subset\mathbb{R}^{k+l}$, let $\ell_v$ denote the line $\mathbb{R}v$ ,
and let $\pi_v:\mathbb{R}^{k+l}\rightarrow\mathbb{R}^{k+l-1}$ denote orthogonal projection to $\ell_v^\bot$.
By $v$-coordinate of a point $p\in\mathbb{R}^{k+l}$ we will mean the scalar projection of $p$ onto $v$. We say that a point $p$ on a $k$-simplex $\sigma$ of $M^k$ in 
$\mathbb{R}^{k+l}$ is an \textit{overcrossing} (respectively \textit{undercrossing}) if there is another point $q\in M$ 
 with $\pi(p)=\pi(q)$ and difference of $v$-coordinate of $p$ and the $v$-coordinate of $q$ is positive (respectively negative).

\section{Cellular moves}
 In this section we describe cellular moves, which we will use repeatedly in the next section to isotope any link to be positive.
 
 Suppose we have a embedded oriented $(k+1)$-disk $D$ in $\mathbb{R}^{k+l}$ such that $D$ meets $M^k$ in a $k$-disk $\sigma$
 in $\partial D$ which is a union of simplices of both $M$ and $\partial D$ and the induced orientations coming from $M$ and $\partial D$ are opposite. Let $M'$ be the manifold obtained from $M$
 by replacing $\sigma$ with $\overline{\partial D\setminus\sigma}$ (with the orientation on the new simplices coming from $\partial D$), Proposition 4.15 of \cite{RS}
 shows that $M$ and $M'$ are ambient isotopic. We call such replacement a \textit{cellular move} along $D$. Hereafter, we will keep calling the manifold $M$ even after applying
  cellular move.
\begin{remark} \label{r2}
 We want the new manifold to be oriented, and so we need the orientations (induced by $\sigma$) on the
co-dimension one faces of $\sigma$, to agree with the induced orientation coming from the new simplices. This forces the
orientation on the new simplices which is why we require the orientations of the simplices common to $M$ and $\partial D$ to be as above.
\end{remark}

   We will use the cellular moves for constructing all our isotopies, they will be of two types:
  \begin{enumerate}
   \item Moving the vertices of $M$ slightly for general position arguments.
   \item Replacing a negative simplex with a union of positive simplices. 
  \end{enumerate}
  
  For the first type of isotopy, we note that for any vertex $x$ of $M^k$, the union of all $k$-simplices of $M$ which contain $x$ is a $k$-cell,
  and slightly moving $x$ is a cellular move. We note that after moving $x$ slightly, a simplex will remain positive (respectively negative) if it was 
  initially positive (respectively negative). We will say more about the second type of isotopy in Remark ~\ref{r4}, after we make a general observation.
  
  The \textit{join} of two subsets $A$ and $B$ of $\mathbb{R}^n$ is defined to be $$A*B:=\{\lambda a+(1-\lambda) b:a\in A, b\in B, \lambda\in[0,1]\}.$$
\begin{lem}\label{A}
 Let $\sigma=[p_0,...,p_k]$ be a $k$-simplex of $M^k$ in general position in $\mathbb{R}^{k+l}$, and suppose we can find a point
$q\in\mathbb{R}^{k+l}$ such that $D=-(q*\sigma)$ (the minus sign indicates that $D$ is oppositely orientated as compared to $q*\sigma$) meets $M$ only in $\sigma$, and $\pi(\mathring{D})$ contains $O$. Then the result of cellular move along $D$
 is that $\sigma$ is replaced by the other simplices of $\partial D$, and all the new simplices are oppositely oriented compared to $\sigma$.
\end{lem}
 
\begin{proof}
 We see that the orientations of all the $k$-faces of $[q,p_0,p_1,...,p_k]$ agree with the orientations of $\sigma$,
 since when expressed in the basis $\pi(p_0),...,\pi(p_k)$, all coefficients of $\pi(q)$ are negative since $O\in\pi(\mathring{D})$. 
 Thus all the new simplices are oppositely oriented compared to $\sigma$ since the induced orientations on new faces come from $-[q,p_0,p_1,...,p_k]$.
\end{proof}

\begin{remark}\label{r3}
In particular, if $\sigma$ was a negative face to begin with, we can isotope $\sigma$ to a union of positive simplices, provided we can find a $q$ 
as in Lemma ~\ref{A}. 
\end{remark}

\begin{remark}\label{r4}
If we choose $q$ to be any point such that $q*\sigma$ meets $M$ only in $\sigma$, and all the coefficients of $\pi(q)$ in the basis $\pi(p_0),...,\pi(p_k)$ are nonzero, then the result of the cellular
move along $-(q*\sigma)$ will be a $k$-link with each simplex in general position (assuming each $k$-simplex of $M$ was already in general position), and the orientations of the new simplices can be read off from the sign of the corresponding coefficient. In particular, if one chooses $q$ such that all the coefficients
are positive, then the orientations of the new simplices after applying the cellular move would be the same as that of $\sigma$.
\end{remark}

\begin{remark}\label{r5}
 We have obtained an alternate way to look at another classical theorem of Alexander (see \cite{A0}), which states that every closed oriented piecewise linear 
 $k$-manifold is a branched cover over $S^k$. Any such manifold $M$ embeds in $\mathbb{R}^{N}$ for some $N>k$, and as we saw above, 
 for a generic orthogonal projection to $\mathbb{R}^{k+1}$, all the simplices will be non-degenerate. For any negative simplex $\sigma$ of $M$ in $\mathbb{R}^{k+1}$, we can choose a
 point $q\in\mathbb{R}^{k+1}$ such that $q*\sigma$ contains $O$ in its interior. Replacing\footnote{Right now, we are just constructing a new piecewise
 linear map, and not saying that this operation is an isotopy. However if $N$ is sufficiently large, by Theorem~\ref{2} we can carry out the entire construction by an isotopy.} $\sigma$ with the other simplices of $-q*\sigma$ gives us a new piecewise linear map from $M$
 to $\mathbb{R}^{k+1}$, with one fewer negative simplex. Thus by induction on the number of negative simplices, we can always construct a map from $M$
 to $\mathbb{R}^{k+1}$ with all simplices being positive, and by Remark~\ref{r1}, we get a piecewise linear branched cover of $M$ over $S^k$ by composing with the pseudo-radial projection.
 It seems likely that this approach will produce a branched cover with fewer number of sheets than Alexander's original construction.
\end{remark}

 The following lemma shows that it is always possible to find embedded disks to do cellular moves if the crossings are only of one type.

\begin{lem}\label{B}  Suppose $f:M\rightarrow\mathbb{R}^{k+l}$ is a embedded closed oriented link, and let $\sigma$ be a $k$-simplex 
of $M$ in general position in $\mathbb{R}^{k+l}$ and does not have both overcrossings and undercrossings. Then there is a
point $q\in\mathbb{R}^{k+2}$ such that $O\in \pi(\mathring{D})$ and $D\cap M=\sigma$, where $D=-(q*\sigma)$.
\end{lem}

\begin{proof} Let us assume that all crossings are overcrossings (respectively undercrossings).
 Choose a point $q\in\mathbb{R}^{k+l}$ such that $O\in \pi(\mathring{D})$ and $\pi(q)\notin\pi(M)$. Note that changing only the $v$-coordinate of $q$ does not
 change the projection $\pi_v(D)$, and we will change the $v$-coordinate of $q$ if necessary.
Let $x\in M\setminus\sigma$ be such that there is a point $y_x\in D$ whose image under $\pi_v$ is the same (since $\pi_v|_{D}$ is injective,
for any given $x$, there can be at most one $y_x$) . If we can
ensure that the difference of $v$-coordinate of $x$ and the $v$-coordinate of of $y_x$ is negative (respectively positive), then we would have
$D\cap M=\sigma$. We can in fact reduce to checking this condition for finitely many such points $x$, as follows:
let $\tau$ be a simplex of $M$, then $\pi_v(\tau)\cap \pi_v(D)$ will be a bounded polytope, hence 
by Proposition 2.7 of \cite{RS}, the convex hull of finitely many points. So as long as we ensure that the $v$-coordinates of all points
which map to these extreme points of $\pi_v(\tau)\cap \pi_v(D)$ satisfy the required inequality, we have that $D\cap \tau=\sigma\cap\tau.$ 
Now, if this holds for all simplices $\tau$ of $M$  then we would have $D\cap M=\sigma$ as required. Since $M$ is compact, there 
are finitely many simplices $\tau$, and thus we only have to check the inequality for finitely many points.

 Now given a point $x\in M\setminus\sigma$ with $\pi_v(x)\in \pi_v(D)\setminus\pi_v(\sigma)$\footnote{Note that if $\pi_v(x)\in\pi_v(\sigma)$, 
then we already know if the crossing at $\pi_v(x)$ is an overcrossing or undercrossing, and this is independent
 of the $v$-coordinate of $q$. This is why we require the condition that $\sigma$ does not have both overcrossings and undercrossings in the statement of
 the lemma.}, let $z$ be the unique point in $\sigma$ whose projection
 under $\pi_v$ is the point of intersection of $\pi_v(\sigma)$ and the line passing through $\pi_v(x)$ and $\pi_v(q)$. Then we will have that $x$ is below (respectively above) $D$ as long as 
 $q$ is above (respectively below) the point where the line $\pi_v(q)+\ell_v$ (i.e.\@ the translate of $\ell_v$ which projects to $\pi_v(q)$) meets the line joining $x$ and $z$. Thus we see that each such point $x$ gives rise
 to a lower (respectively upper) bound of $v$-coordinate of $q$, and we can simultaneously satisfy finitely many such bounds. The result follows.
 \end{proof}

\begin{remark}\label{r6}
  Sometimes we will not be able to find a $q$ as in Lemma ~\ref{B}, but we may be able to subdivide $\sigma$ into cells so that the crossings in each subcell
is only of one type and then we have similar results as Lemmas \ref{B} and \ref{A}.
Suppose $f:M\rightarrow\mathbb{R}^{k+2}$ is a embedded closed oriented link, and let $\tau$ be 
a $k$-dimensional cell contained in a negative $k$-simplex $\sigma$ of $M$ in $\mathbb{R}^{k+2}$.
If $\tau$ does not have both overcrossings and undercrossings, then there is a point $q\in\mathbb{R}^{k+2}$ such that $D=-(q*\tau)$
meets $M$ only in $\tau$, and $\pi(\mathring{D})$ contains $O$. Moreover, the result of cellular move along $D$ is that $\tau$ is replaced by a union
of positive simplices.
\end{remark}


\section{Co-dimension two braiding} In the first subsection, we will use the tools developed so far to complete the proof of Theorem~\ref{1}.
We ask some questions about co-dimension two braidings in other cases in the second subsection.
We observe that since we have co-dimension $l=2$, then $\ell=\ell_v$  and $\pi=\pi_v$.

\subsection{Isotoping co-dimension two link to be positive}
 To prove our main result it remains to show the following.
\begin{thm}\label{1a}
 For $1\leq k\leq 3$ , each embedded closed oriented link $f:M^k\rightarrow\mathbb{R}^{k+2}$ is isotopic to a positive link.
\end{thm}
\textit{Strategy  of proof}. We will use induction on the number of ``negative $k$-simplices''. If all crossings are of one type then we can
use cellular moves to replace (isotope) a negative $k$-simplex with a number of positive $k$-simplices. Sometimes we will have to break up a negative
$k$-simplex into smaller $k$-simplices (temporarily increasing the number of negative $k$-simplices) and show that we can use cellular moves to replace each of the subsimplices,
thereby reducing the number of negative $k$-simplices.

\textit{Notation}.  Let $S$ be a subset of $M$. We say that a point $x\in S$
is a \textit{double point} of $S$ if $|\pi|_S^{-1}(\pi(x))|\geq 2$, a \textit{triple point} of $S$ if $|\pi|_S^{-1}(\pi(x))|\geq 3$,
a \textit{quadruple point} if $|\pi|_S^{-1}(\pi(x))|\geq 4$, and a \textit{quintuple point} of $S$ if $|\pi|_S^{-1}(\pi(x))|\geq 5$. We call the collection of 
all double (respectively triple) points of a subset $S$ of $M$ the double (respectively triple) point set of $S$ and denote it by 
$\mathcal{D}_S$ (respectively $\mathcal{T}_S$), and we call their closure in $S$  
the double (respectively triple) point complex of $S$ and denote it by $\overline{\mathcal{D}_S}$ (respectively $\overline{\mathcal{T}_S}$).
If $S$ is not mentioned explicitly, it is understood that $S$ is $M$.
For any $k$-simplex $\sigma$ of $M$, let $\mathscr{T}_\sigma$ denote the closure in $\sigma$ of $\sigma\cap \mathcal{T_M}$.

\textit{Figures}. A note on the figures; in the cases $k=1,2$, when we are illustrating special cases of crossings on negative $k$-simplices
we will frequently show both an immersed picture, where we will show all the simplices crossing, and a preimage picture, where we indicate all
the crossing points in the negative $k$-simplex. In the case $k=3$, we can only draw preimage pictures.

\begin{proof}[Proof of Theorem ~\ref{1a}] We argue the 3 cases for $k$ separately.

\textbf{Case: $k=1$} (Alexander, \cite{A}). The proof is by induction on the number of negative $1$-simplices of the triangulation of $M$.

\textit{General Position}. We can ensure the all the crossings are isolated double points, and there are no triple points.

 \textit{Special Case}. If a negative $1$-simplex does not have both overcrossings and undercrossings, then we can use Lemma ~\ref{B}
 to replace the $1$-simplex with positive $1$-simplices.
 
 \begin{figure}[!ht]
  \includegraphics [width=10cm] {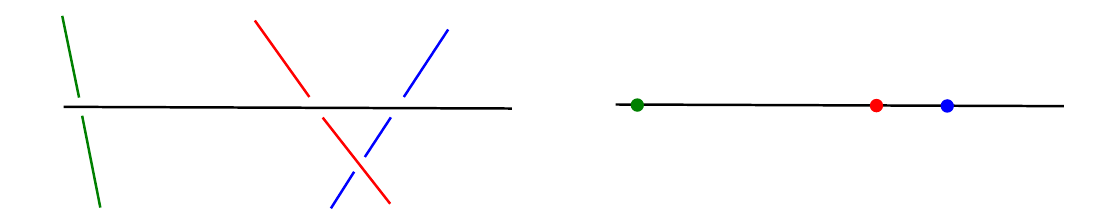} 
  \caption{Immersed Pictures of 1-simplices intersecting. Preimage picture of the black $1$-simplex, where the crossing points are shown.}
 \label{F1}
 \end{figure}
\textit{General Case}. We can break up a negative $1$-simplex into smaller $1$-simplices such that no part has both overcrossings and undercrossings,
and we can apply Lemma ~\ref{B} to each of the subsimplices.
We note that applying such a cellular move to one such subsimplex does not introduce any new crossings on the other subsimplices of our original 
 negative $1$-simplex. Thus we can reduce the number of negative $1$-simplices, and we are done by induction for the case $k=1$.

 \textbf{Digression}. For $k=2,3$ we have to deal with the fact that if we break up a $k$-simplex and apply cellular move to the various parts, the result will not be triangulated
 any more. Of course we could subdivide the adjacent $k$-simplices so that the result is in fact triangulated, but this may increase 
 the number of negative $k$-simplices, which we do not want to happen. We will need to modify the induction hypothesis in cases $k=2,3$.
  For this reason, following Kamada (see Chapter 26 in \cite{K2}), we will introduce the notion of a division of a piecewise linear manifold.
  
 A \emph{division} for a link $M^k\subseteq\mathbb{R}^{k+2}$ is a collection of $k$-simplices $\{\sigma_1,...,\sigma_l\}$ whose union is $M$,
 and such that for distinct $i$ and $j$, if $\sigma_i\cap \sigma_j$ is nonempty, it is contained in faces of both $\sigma_i$ and $\sigma_j$, and is a face of $\sigma_i$
 or $\sigma_j$. 
 We say that the $\sigma_i$'s are $k$-simplices of the division for $M$, and the 
 notion of a positive/negative $k$-simplex is defined as before. We say that a $k$-simplex $\sigma$ is \textit{inner} (respectively \textit{outer}) if its intersection with any other $k$-simplex $\tau$ is a 
 face of $\sigma$ (respectively $\tau$). 
  \begin{figure}[!ht]
  \includegraphics [width=6cm] {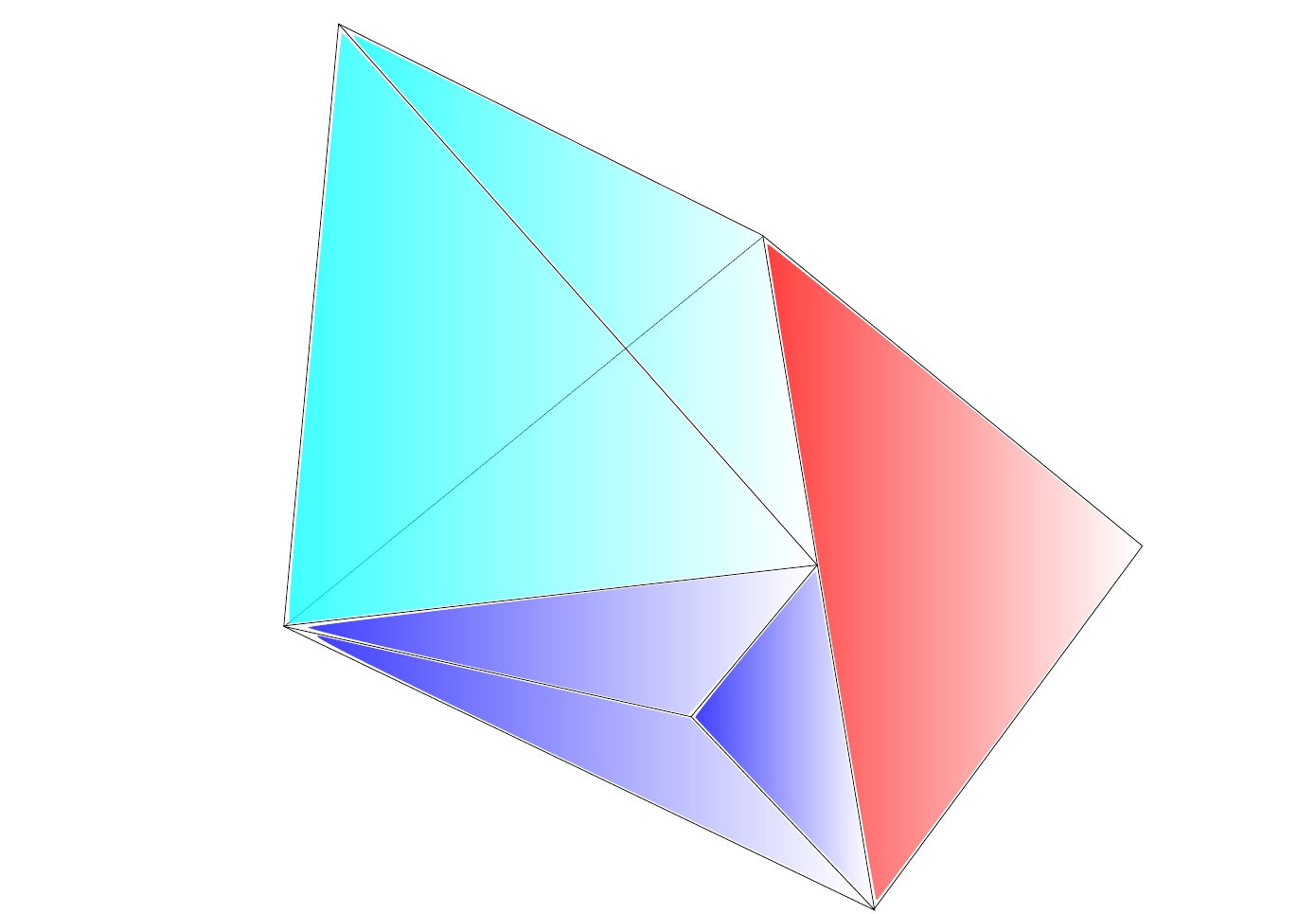} 
  \caption{A part of a division, where the blue simplices are inner, and the red is outer.}
  \label{F2}
\end{figure}
If $\sigma$ is inner, then we can break $\sigma$ up into smaller cells and apply cellular moves on each subcell, and the result would be a division.
We can only move a vertex $x$ of a $k$-simplex $\sigma$ of $M$ slightly without changing the number of $k$-simplices of the division if $x$ is a vertex
of every $k$-simplex $\tau$ that contains $x$. In particular if $\sigma$ is outer, we can slightly move any of its vertices slightly without changing the number of $k$-simplices of the division.
 A division is a triangulation if and only if all $k$-simplices are both inner and outer. We say a division is \textit{good} if all 
the negative $k$-simplices are outer. Lemmas ~\ref{A} and  ~\ref{B} still hold if we have a division of $M$.

\textit{Notation}. For any $k$-simplex $\sigma$ of $M^k$, let $X(\sigma)$ denote the union of all the $k$-simplices which are not adjacent to or equal to $\sigma$,
let $Y(\sigma)$ denote the complement of $\sigma$ in $M$, let $\overline{Y(\sigma)}$ denote the closure of $Y(\sigma)$ in $M$, and let $Z(\sigma)$ denote 
the union of all the $k$-simplices of $M$ whose intersection with $\sigma$ in has dimension at most $k-2$.
 
 We return to the proof of Theorem~\ref{1a}.

\textbf{Case: $k=2.$} The proof is by induction on the number of negative $2$-simplices of the division of $M$.

\textit{General Position for the initial triangulation}. 
We may assume that all the crossings are double point lines and isolated triple points in the interior of respective $2$-simplices,
and there are no quadruple points.
\begin{proof}
We will make modifications to $M$ in three steps, at each step assuming results from the previous steps hold. In the first step we can make sure that all $2$-simplices intersect ``nicely'' pairwise, and in the last
 step we will make sure that all triples of $2$-simplices intersect ``nicely'', after which we get the required general position statement. Step 2 is a special case of Step 3, where we make sure all non-adjacent 
 triples of $2$-simplices intersect ``nicely''. By slightly moving each vertex of (the triangulation of) $M$, we may assume 
that:
\begin{enumerate}

 \item The projection of a vertex $x$ of (the initial triangulation of) $M$ is not contained in a hyperplane generated (=affinely spanned) by the projection a 2-simplex $\tau$ of $M$
 , if $x\notin \tau$.
 
If $y_1,y_2,y_3$ are vertices of $M$ such that $\pi(x)$ is contained in the hyperplane $H$ defined by $\pi(y_1),\pi(y_2),\pi(y_2)$
 then that means $\alpha(\pi(x))=0$ where $\alpha$ is the dual (with respect to the standard inner product) linear functional defined by choosing an unit normal to $H$ in $\pi(\mathbb{R}^4)=\mathbb{R}^3$. By a slight perturbation of
 $x$ one can ensure that $\alpha(\pi(x))\neq 0$, and moreover this is generic, i.e.\@ a small perturbation of $x$ would not change the non-vanishing
of $\alpha(\pi(x))$. We can keep on perturbing the vertices slightly until the above condition holds.

 This ensures that the projection of two 2-simplices can only intersect in a line segment, that the projection of two 2-simplices  which only share a vertex cannot
 intersect along any edge of either 2-simplex, that the projection of two 2-simplices  which  share an edge do not intersect elsewhere.
 Consequently, $\overline{\mathcal{D}_M}$ and $\pi(\overline{\mathcal{D}_M})$ are graphs.
 \item For any 2-simplex $\sigma$ of $M$, we have that $\pi(\overline{\mathcal{D}_{X(\sigma)}})$ meets $\pi(\sigma)$ and $\pi(\partial\sigma)$ 
 transversely. 
 
 We will only perturb the vertices of $\sigma$, and so $\pi(\overline{\mathcal{D}_{X(\sigma)}})$ will stay fixed.
 Let $[p_1,p_2]$ be an edge of $\pi(\overline{\mathcal{D}_{X(\sigma)}})$. If we make sure that the points $p_1,p_2$ are not in the hyperplane generated
 by $\pi(\sigma)$, and the projections of vertices $q_1,q_2$ and $q_3$ of $\sigma$ are maximally affinely independent with $p_1,p_2$ 
 (i.e.\@ there is no hyperplane in $\mathbb{R}^3$ containing $\sigma$ and $[p_1,p_2]$, or equivalently $p_2-p_1$, $\pi(q_2)-\pi(q_1)$ and $\pi(q_3)-\pi(q_1)$
 are linearly independent, or equivalently the determinant of $[p_2-p_1| \pi(q_2)-\pi(q_1)|\pi(q_3)-\pi(q_1)]$ is nonzero),
 then $[p,q]$ and $\sigma$ have the required property. If we move the vertices of $\sigma$
   slightly, this property still remains true, hence we can make sure that the required property holds for each edge of $\pi(\overline{\mathcal{D}_{X(\sigma)}})$. 
Slightly perturbing the vertices of $M$ would not change this, and hence we can make sure the property holds for all 2-simplices of $M$. 

 \item For any $2$-simplex $\sigma$ of $M$, we have that $\pi(\overline{\mathcal{D}_{Y(\sigma)}})$ meets $\pi(\sigma)$ and $\pi(\partial\sigma)$
 transversely, except at projection of points of $\partial\sigma$ which are not triple points. 
 
 Given any vertex $x$ of $M$, we can make sure that the set of normal vectors (based at $\pi(x)$) in $\mathbb{R}^3$ to the hyperplane generated by projections of all the 2-simplices that have $x$ has a vertex are maximally affinely independent (i.e.\@ if we think of $\pi(x)$ as the origin, any three of these normal vectors are linearly independent) by perturbing other (i.e.\@ except $x$) vertices of such $2$-simplices. This condition ensures that for any three 2-simplices that share the vertex $x$, their projection can intersect in at most one point.
 We can make sure the above condition holds for all vertices $x$ of $M$. For any 2-simplex
 $\sigma$ of $M$, we can make sure that there is no triple point in any edge of $\sigma$ by perturbing the vertices of $\sigma$ slightly, while fixing the hyperplane generated by the projection of $\sigma$.
 
\end{enumerate} 
Thus, the projection of any triple of 2-simplices intersect at at most one point, and triple points occur in the interior of their respective simplices. 
\end{proof}

\textit{General Position for a division}. When applying a cellular move along $D=-(q*\nu)$, we may assume that
the $q$ and $\nu$ are chosen so that $\pi(\overline{\mathcal{D}_{M\setminus{\mathring{\nu}}}})$ meets  $\pi(\partial D)$ and $\pi(\delta D)$ transversely.
 Consequently, there are no quadruple points, and moreover all triple points are isolated and lie in the interior of their respective 2-simplices.

\textit{Special Cases}. Let us look at some special cases (which will contain previous special cases) of crossings in a negative $2$-simplex:
 \begin{figure}[!ht]
  \includegraphics [width=11cm] {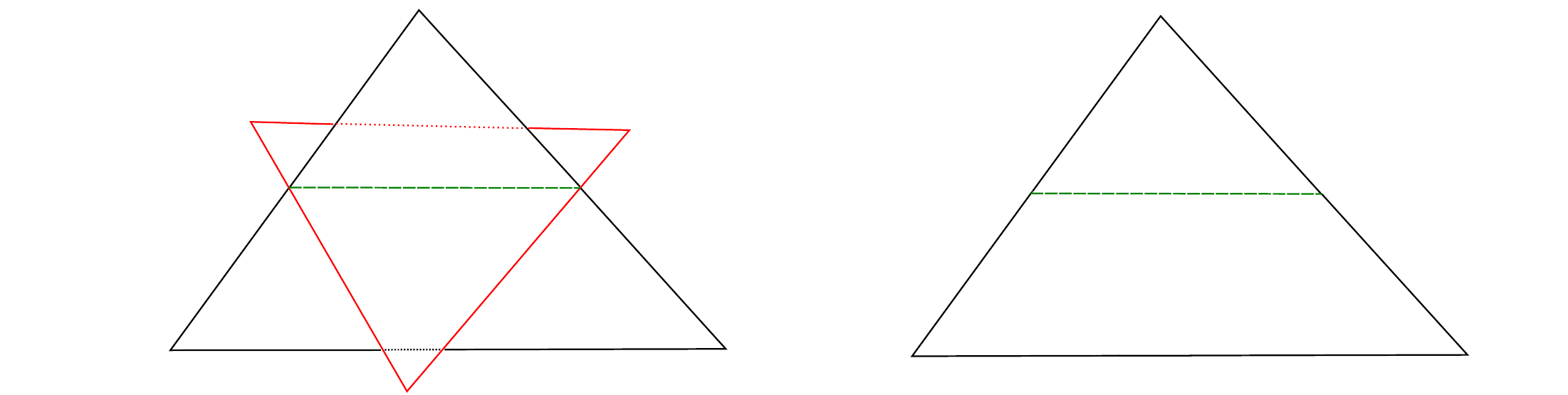} 
  \caption{Immersed and Preimage Pictures: two non adjacent $2$-simplices intersecting in a double point line, illustrating special case 1.}
 \label{F3}
 \end{figure}

 \begin{figure}[htbp]
  \includegraphics [width=10cm] {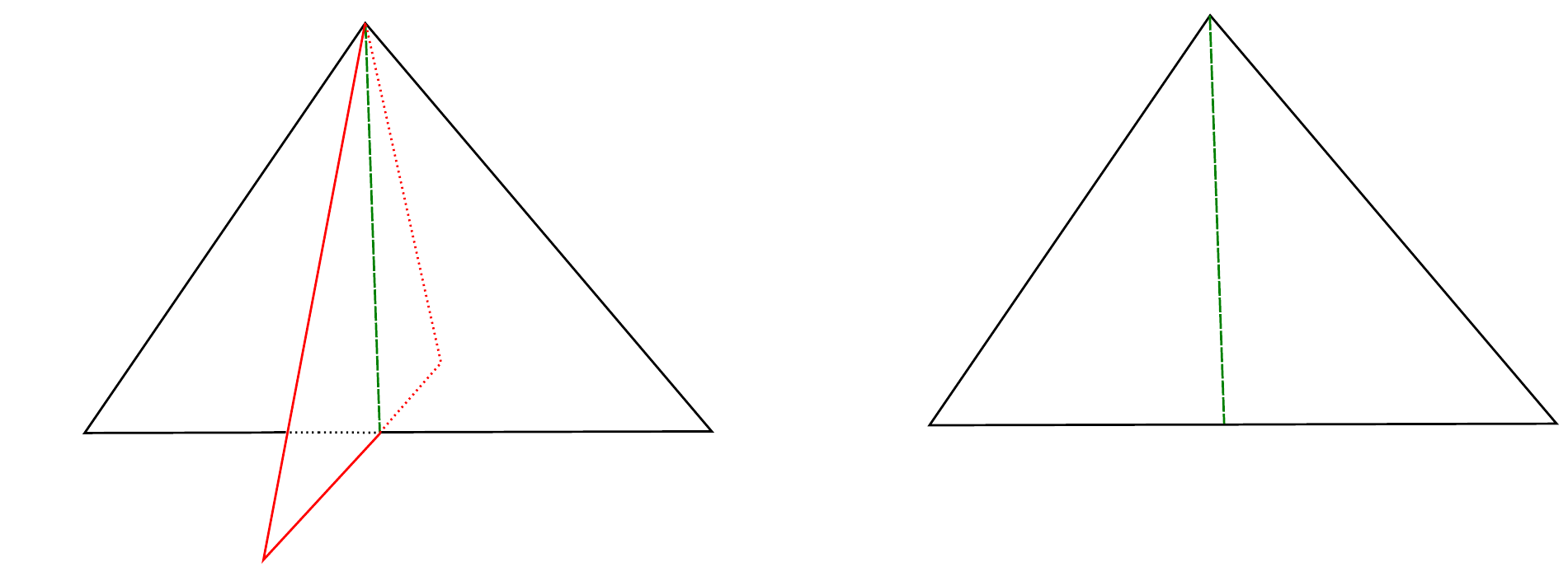} 
  \caption{Immersed and Preimage Pictures: two $2$-simplices sharing a vertex intersecting in a double point line, illustrating special case 1.}
 \label{F4}
 \end{figure}

 \begin{enumerate}
  \item  Suppose $\sigma$ is a negative $2$-simplex such that does not have both overcrossings and undercrossings, see Figures ~\ref{F3} and ~\ref{F4}.
  
  We can replace $\sigma$ with a union of positive $2$-simplices by Lemma ~\ref{B}.

  \item  Suppose $\sigma$ is a negative $2$-simplex such that there are no triple points, see Figure ~\ref{F5}.
  
  \begin{figure}[!ht]
  \includegraphics [width=5cm] {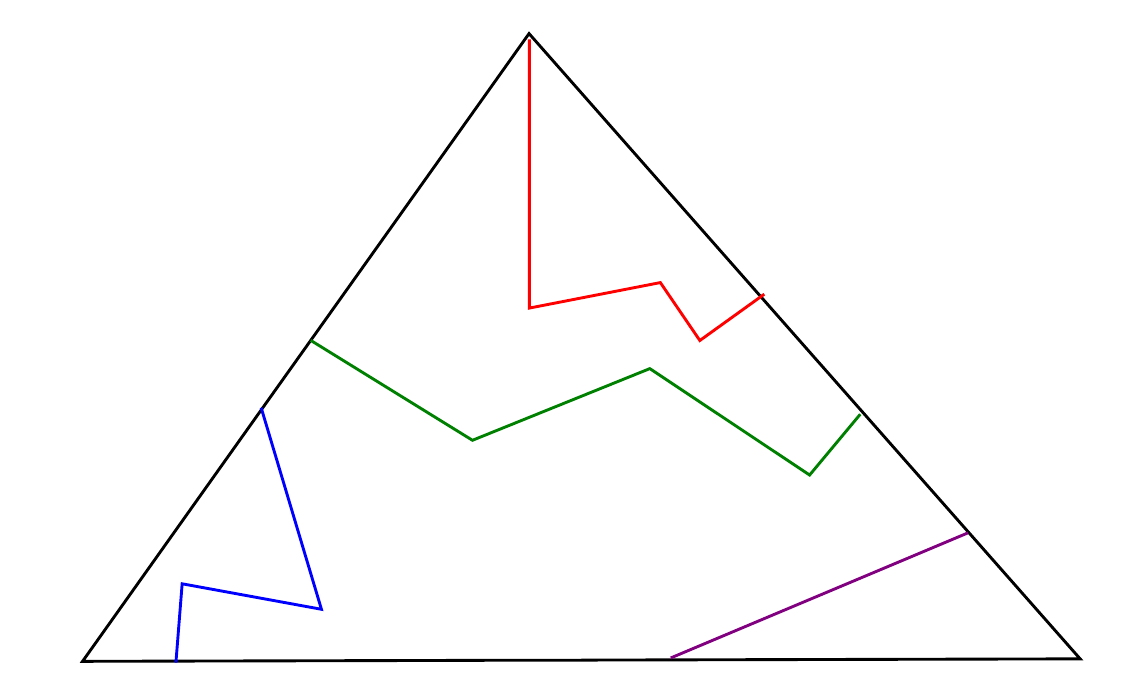} 
  \caption{Preimage Picture: crossings without triple points, illustrating special case 2.}
  \label{F5}
  \end{figure}  
  In this case we can break $\sigma$ up into smaller $2$-simplices which are inner, and moreover each of the subsimplices
  does not have both overcrossings and undercrossings, and we can use Lemma ~\ref{B} to replace each of them with positive $2$-simplices. So we can reduce the number of negative $2$-simplices by one.

  \item Suppose $\sigma$ is a negative $2$-simplex with exactly one triple point $p\in \sigma$, see Figure ~\ref{F6}.
  
  \begin{figure}[htbp]
  \includegraphics [width=11cm] {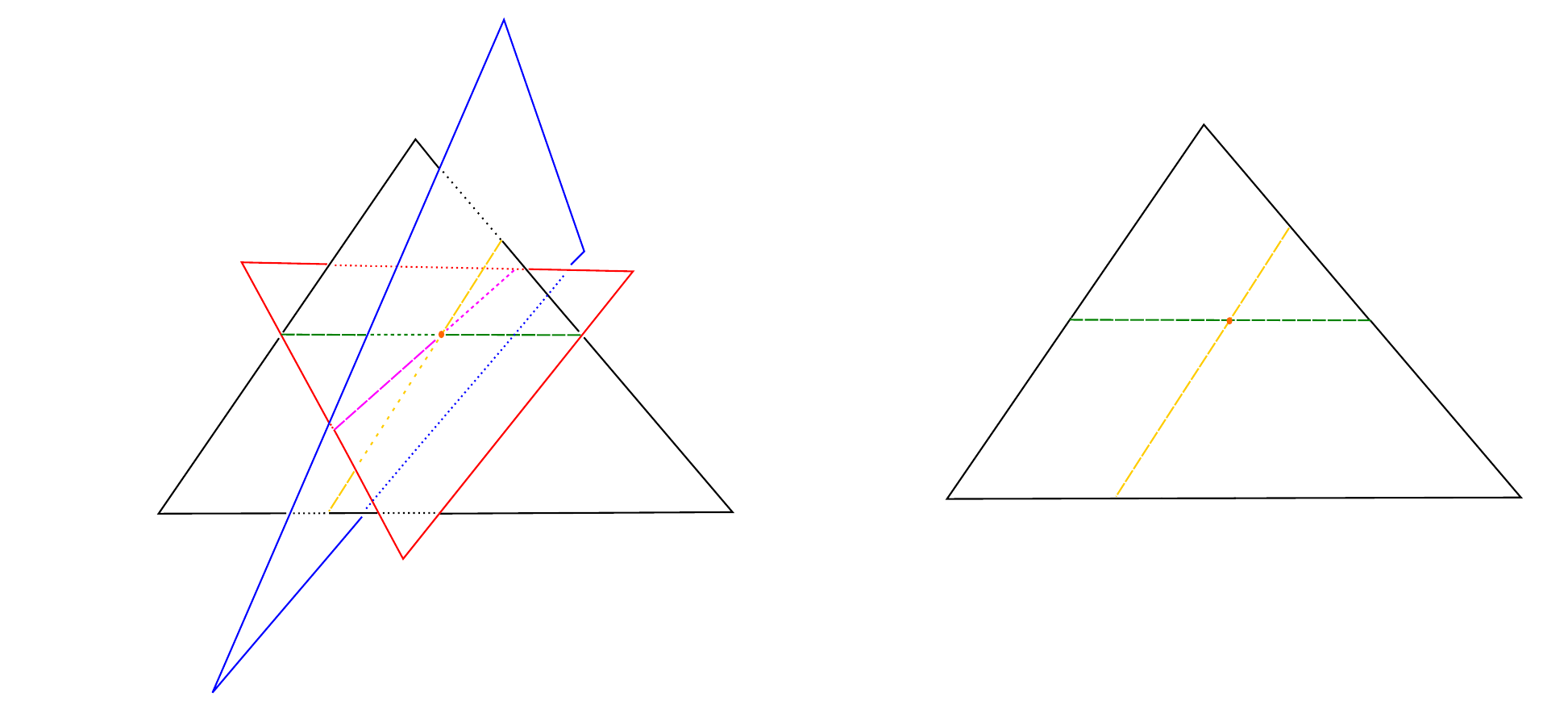} 
  \caption{Immersed and Preimage Pictures: three non adjacent simplices intersecting in an isolated triple point, illustrating special case 3.}
 \label{F6}
 \end{figure}
  We know that the line segment $[O,\pi (p)]$ can meet the projection of each of the three 2-simplices giving rise to the triple point only in $\pi(p)$ (since all the $2$-simplices
  are in general position), and we choose a point $q\in\mathbb{R}^4$ (the $v$-coordinate will be changed later if necessary) such that $O$ is in the line segment $(\pi (q),\pi (p))$.
  As in the proof of Lemma ~\ref{B}, by choosing the $v$-coordinate of $q$ to be sufficiently positive (or negative), we have $[q,p]\cap M= \{p\}$.
  \begin{figure}[htbp]
  \includegraphics [width=11cm] {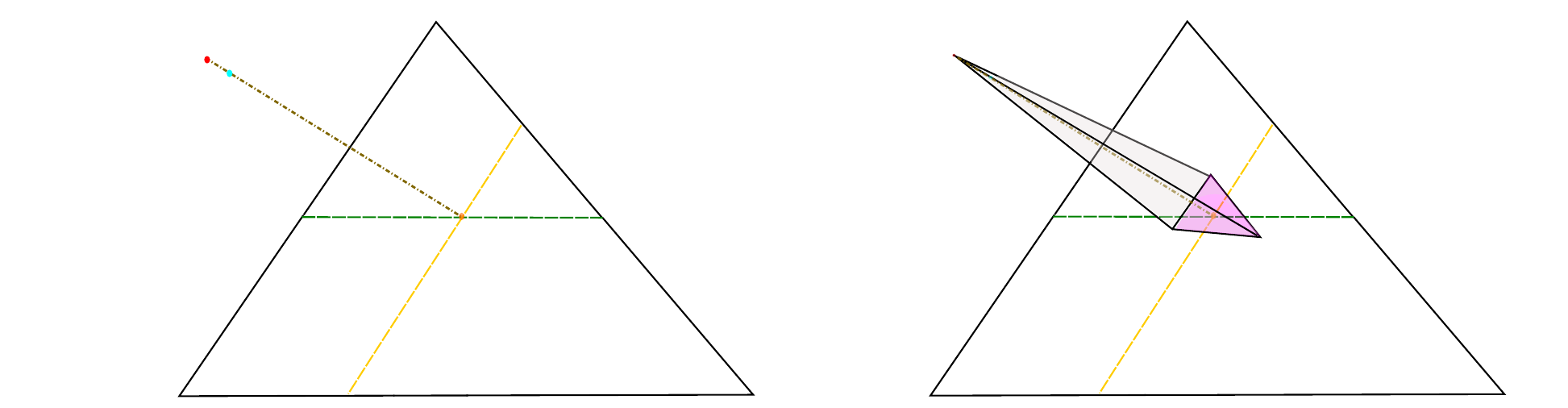} 
 \caption{Replacing a small neighborhood of an isolated triple point by cellular move.}
 \label{F7}
 \end{figure}

   By compactness we have $d([q,p],\overline{Y(\sigma)})>0$, and hence we can choose a small 2-simplex $[p_0,p_1,p_2]$ in $\sigma$
  containing $p$ such that $[q,p_0,p_1,p_2]$ meets $M$ only at $[p_0,p_1,p_2]$, and we can use cellular move to replace $[p_0,p_1,p_2]$ by
  positive $2$-simplices, see Figure ~\ref{F7}. The rest of $\sigma$ can be broken up into smaller inner $2$-simplices each of which does not have both overcrossings and undercrossings,
  and by Lemma ~\ref{B}, we can replace them with positive $2$-simplices and hence we are done.
 \end{enumerate}

\textit{General Case}. Suppose $\sigma$ is any negative $2$-simplex, see Figure ~\ref{F8}.
\begin{figure}[!ht]
  \includegraphics [width=5cm] {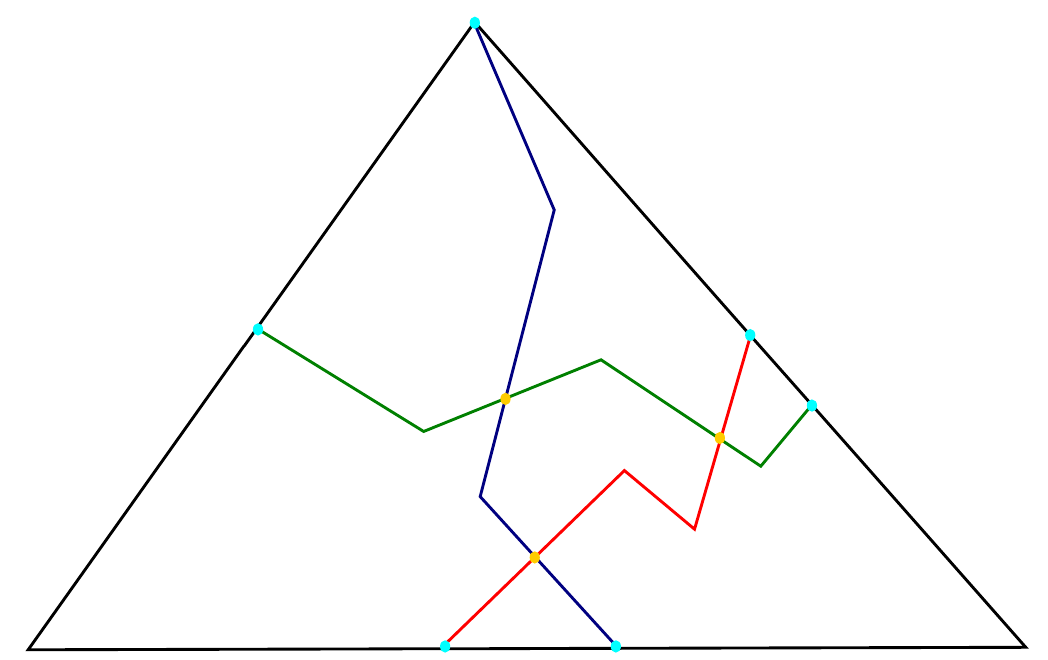} 
  \caption{Preimage Picture: crossings with triple points, illustrating the general case.}
  \label{F8}
  \end{figure}
  
  We break $\sigma$ up into smaller inner $2$-simplices each of which has at most one interior triple point and then
  use the above special cases. Thus we can reduce 
 the number of negative simplices, and we are done by induction for the case $k=2$.

\textbf{Case: $k=3.$}
It suffices to prove the following claim, since the initial triangulation is a good division.

\begin{claim}
 Every embedded closed oriented link $M^3$ in $\mathbb{R}^5$ with a good division is isotopic to a positive link.
\end{claim}
 The proof of the claim is by induction on the number of negative $3$-simplices on the good division.
\begin{remark}
 After we prove the claim, it follows that the result holds for any division, because we can subdivide any division further so that it becomes a good division.
\end{remark}

\textit{General Position}. We may assume that the double point complex is a 2-dimensional CW-complex, the triple point complex is a graph,
all quadruple points are isolated and in the interior of respective 3-simplices, and there are no quintiple points. Moreover we can also assume that all
 triple points are disjoint from $1$-faces of $3$-simplices, and for any vertex $x$ of $\overline{\mathcal{T}_M}$, the set $\pi|_M^{-1}(\pi(x))$ contains exactly one point
 on the 2-faces of $3$-simplices of $M$.

\begin{remark} This can be proved in a similar way we proved the general position statement in the case $k=2$, and we outline an argument below. 
By slightly moving each vertex of (the triangulation of) $M$, we may assume that for the initial triangulation we have:
\begin{enumerate} 
 \item The projection of a vertex $x$ of $M$ is not contained in a hyperplane generated by the projection a 3-simplex $\tau$ of $M$, if $x\notin \tau$. Consequently,
 $\overline{\mathcal{D}_M}$ and $\pi(\overline{\mathcal{D}_M})$ are 2-dimensional CW-complexes.
   
 \item For any $3$-simplex $\sigma$ of $M$, $\pi(\overline{\mathcal{D}_{Y(\sigma)}})$ and the edges of $\pi(\overline{\mathcal{D}_{Y(\sigma)}})$  meets 
 $\pi(\sigma)$ and $\pi(\partial\sigma)$ transversely, except at projection of points of $\partial\sigma$ which are not triple points.
  Hence, $\overline{\mathcal{T}_M}$ and $\pi(\overline{\mathcal{T}_M})$ are graphs, and 
  for any vertex $x$ of $\overline{\mathcal{T}_M}$, the set $\pi|_M^{-1}(\pi(x))$ contains exactly one point on the $2$-faces
  of $3$-simplices of $M$.
  \item For any $3$-simplex $\sigma$ of $M$, $\pi(\overline{\mathcal{T}_{Y(\sigma)}})$
  meets $\pi(\sigma)$ and $\pi(\partial\sigma)$ transversely, except at projection of points of $\partial\sigma$ which are not quadruple points. 
 \end{enumerate}
 
Now we have the required general position statement for the initial triangulation. 
At each step of applying cellular move along $D=-(q*\nu)$, we may assume that
  $\nu$ is chosen so that there are no vertices of $\overline{\mathcal{T}_{M}}$ in the $2$ faces of 
  $\nu$ except the case that such a point is in $\overline{\mathcal{T}_M}\setminus\mathcal{T}_M$,
 $q$ is chosen so that $\pi(\overline{\mathcal{D}_{M\setminus{\mathring{\nu}}}})$ and $\pi(\overline{\mathcal{T}_{M\setminus{\mathring{\nu}}}})$ meets  
 $\pi(\partial D)$ and $\pi(\delta D)$ transversely. 
 We will then have the required general position statement.
 \end{remark}

\textit{Special Cases}. 
We will look at some special cases (which typically will contain previous special cases) of crossings in a negative $3$-simplex: 
 \begin{enumerate}
  
  \item  Suppose $\sigma$ is a negative $3$-simplex such that all crossings are overcrossing (respectively undercrossing).
  
  We can replace $\sigma$ with a union of positive $3$-simplices by Lemma~\ref{B}.
  
  \item  Suppose $\sigma$ is a negative $3$-simplex such that there are no triple points.
  
  In this case we can break $\sigma$ up into smaller inner $3$-simplices such that the crossings are only overcrossing or undercrossing (but not both), and we
  can use Lemma~\ref{B} to replace each of them with positive $3$-simplices. So we can reduce the number of negative $3$-simplices by one.
  
  \item Suppose $\sigma$ is a negative $3$-simplex with exactly one triple point line segment $[p_0, p_1]$ (with $p_0$ and $p_1$ is not in a vertex or edge of $\sigma$)
  coming from $3$-simplices $\tau$ (above $\sigma$) and $\eta$ (below $\sigma$)
  \footnote {Note that if both $\tau$ and $\eta$ are above (respectively below) $\sigma$, then the crossings in $\sigma$ corresponding to $\tau$ and $\eta$ are both 
   undercrossings (respectively overcrossings), and we are in special case 2.}
   which are not adjacent to $\sigma$, and such that $\pi(\tau)\cap \pi(\eta)$ contains the projection of $[p_0, p_1]$ in its interior, 
   and there are no quadruple points in $\sigma$, see Figure ~\ref{F9}.  
  \begin{figure}[!ht]
  \includegraphics [width=6cm] {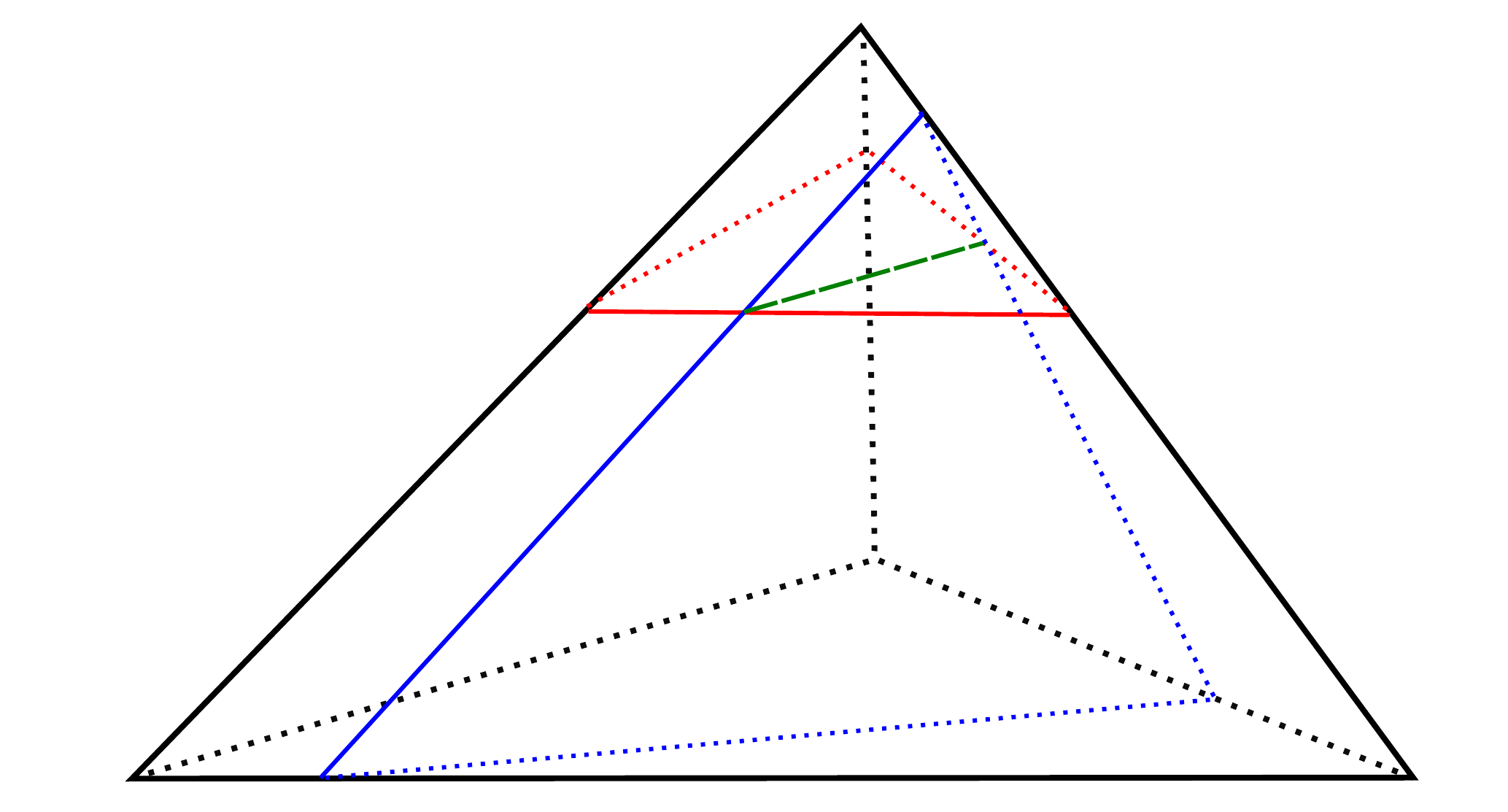} 
  \caption{Preimage Picture: a $3$-simplex intersecting with two non adjacent $3$-simplices in a triple point line segment, illustrating special case 3.}
  \label{F9}
  \end{figure}
  
  Since the $3$-simplices $\tau$ and $\eta$ are in general position the $2$-simplex $[O,\pi(p_0),\pi(p_1)]$ meets $\pi(\tau)$
  and $\pi(\eta)$ only in $[\pi(p_0),\pi(p_1)]$. We choose a point $q\in \mathbb{R}^5$ (the $v$-coordinate will be changed later if necessary) 
   such that $O$ is in the interior of $[\pi(q),\pi(p_0),\pi(p_1)]$,
  and consequently the $2$-simplex $[\pi(q),\pi(p_0),\pi(p_1)]$ meets $\pi(\tau)$ and $\pi(\eta)$ only in $[\pi(p_0),\pi(p_1)]$.
  The hypotheses ensures that in $[\pi(q),\pi(p_0),\pi(p_1)]$ we do not see other (i.e.\@ except $[\pi(p_0),\pi(p_1)]$) lines of over or under crossings starting from the
   vertices $\pi(p_0),\pi(p_1)$.
  Just like in the proof of Lemma~\ref{B}, we may choose the $v$-coordinate of $q$ to be sufficiently positive (or negative) such that:
  \begin{enumerate}  
   \item If $\rho$ is a $3$-simplex whose intersection with $\sigma$ is $2$ dimensional, then the cone $D$ meets $\rho$ only in $\sigma\cap\rho$.
   \item  $[q,p_0,p_1]$ does not intersect $Z(\sigma)$.
 
  \end{enumerate}
  By compactness, the distance between $[q,p_0,p_1]$ and $Z(\sigma)$ is positive, and hence we can choose a
  cell $\nu$ in $\sigma$ containing $[p_0,p_1]$ such that $q*\nu$ meets $M$ only in $\nu$. 
  Using Remark ~\ref{r6}, we can use a cellular move to replace $\nu$ with a finite union of positive $3$-simplices, and can break the rest of $\sigma$ into 
  smaller inner $3$-simplices and use Lemma~\ref{B} to replace them with positive $3$-simplices. So we can reduce the number of negative $3$-simplices by one.
  
  \item Suppose $\sigma$ is a negative $3$-simplex such that $\mathscr{T}_\sigma$ is non empty but does not meet $\sigma$ in a vertex or an edge, see Figure ~\ref{F10}.
  
  \begin{figure}[!ht]
  \includegraphics [width=6cm] {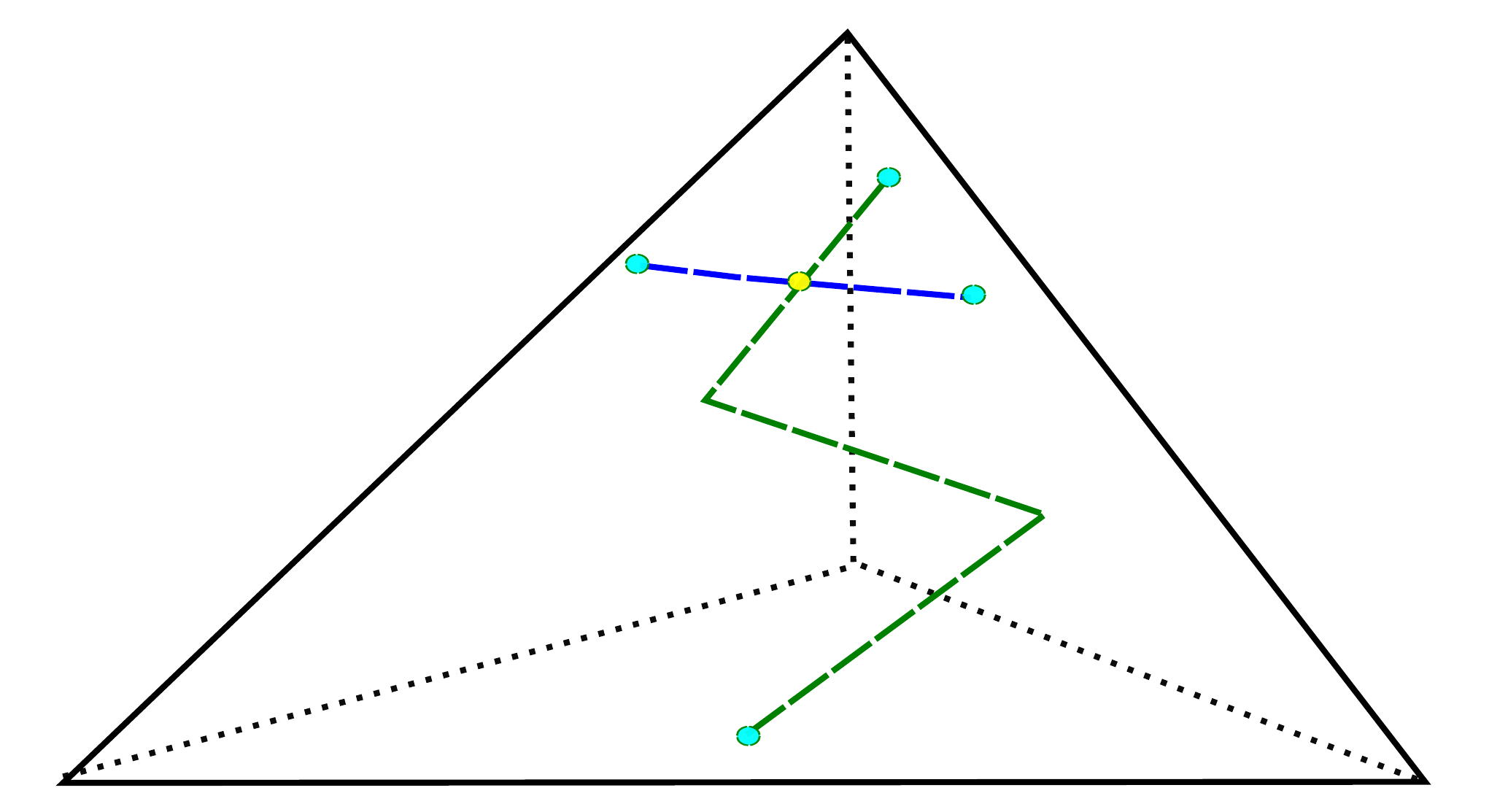} 
  \caption{Preimage Picture: a $3$-simplex with triple point lines meeting at a quadruple point, illustrating special case 4 (double points not indicated).}
  \label{F10}
  \end{figure}
  There will be finitely many points $p_1,...,p_m$ in the interior of $\sigma$ which are either a quadruple point or a vertex of the triple point complex $\overline{\mathcal{T}_M}$.
  We can choose points $q_1,....,q_m$ such that $O\in(\pi(q_i),\pi(p_i))$ and each of these line segments $[q_i,p_i]$ are mutually disjoint. Like before,
  we can find $3$-simplices $P_i$ containing $p_i$ in $\sigma$ such that $q_i*P_i$ are mutually disjoint, and then we can use cellular moves to replace
  each $P_i$ with union of positive $3$-simplices. The rest of $\sigma$ can be broken up into smaller inner $3$-simplices such 
  that we are in the previous special cases. The hypothesis and our general position statement 
 ensures that for all subsimplices which contain triple point line segments, we are in the previous
  special case. As we have seen, we can replace each of these $3$-simplices by positive $3$-simplices, and hence we can reduce the number
  of negative $3$-simplices by one.
  
  The special cases of crossings in negative $3$-simplices we considered so far are analougus to the ones we saw in the case $k=2$.
  In next two special cases we will consider the ``new'' type of crossings, when $\mathscr{T}_\sigma$ meets a vertex or an edge of $\sigma$, and we will
  need a new idea.
  
  \item Suppose $\sigma$ is a negative $3$-simplex, with the only crossings coming from $3$-simplices $\tau$ (above) and $\eta$ (below) who share a vertex $p_0$ with $\sigma$.
  Moreover, there is triple point semiopen line segment $(p_0,p_1]$ in $\sigma$, and $\pi(\tau)\cap \pi(\eta)$ contains $\pi(p_0,p_1]$ in its interior, see Figure ~\ref{F11}.
   \begin{figure}[!ht]
  \includegraphics [width=6cm] {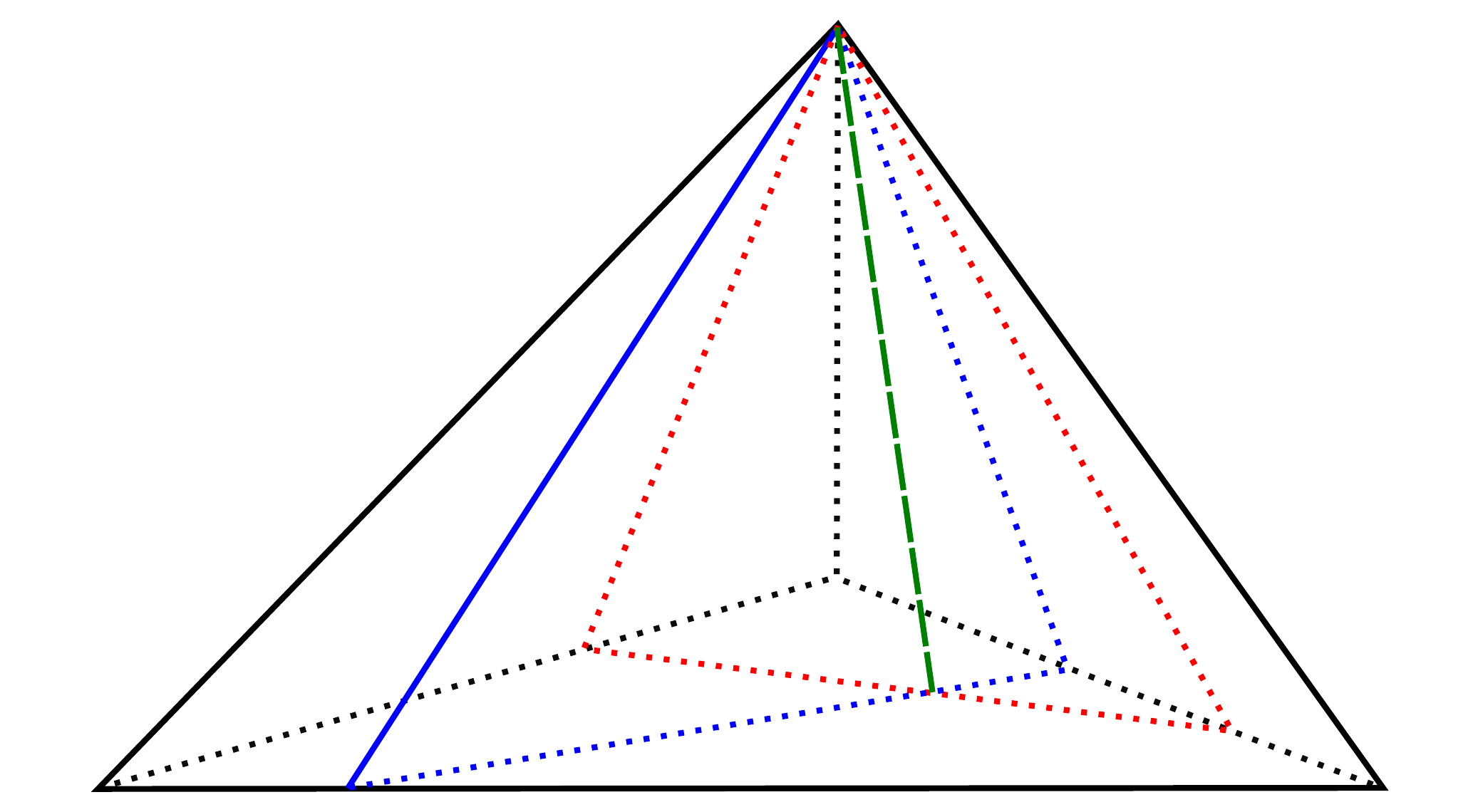} 
  \caption{Preimage Picture: three adjacent $3$-simplices sharing a vertex intersecting in a triple point semiopen line segment, illustrating special case 5.}
  \label{F11}
  \end{figure}
  
  Let $\sigma_1$, $\sigma_2$ be the subcells of $\sigma$ such that the hyperplane generated by $\pi(\tau)$ breaks $\pi(\sigma)$ into the two parts $\pi(\sigma_1)$ and
  $\pi(\sigma_2)$, and let us assume that $\pi(\sigma_1)$ is in the same half-space as $O$.
  As in the proof of Lemma ~\ref{B}, we can choose a point (by making the $v$-coordinate sufficiently positive) $q$ such that the cone $q*\sigma_1$  
  meets $M$ only in $\sigma_1$, and $O$ is in the interior of $\pi(q*\sigma_1)$. We use a cellular move to replace $\sigma_1$ by the other faces of this cone.  
  If $\tau$ is negative, it must be a outer $3$-simplex (since we assumed that the division is good), and hence we can move some of the other (except $p_0$) vertices of $\tau$ a little (so that the projection 
  of the vertices lie in the half-space generated by the old $\pi(\tau)$, containing $O$) so that $\sigma_2$ does not have any triple point.
  If $\tau$ is positive, we can apply a cellular move on a smaller subsimplex (so that all the new $3$-simplices are positive) of $\tau$ containing the triple points, so that $\sigma_2$ 
  does not have any triple point. By using the above special cases, we see that we can replace $\sigma_2$ with a union of positive $3$-simplices, 
  and we have reduced the number of negative $3$-simplices by one.
  
  A similar argument works in the special case:
  \item Suppose $\sigma$ is a negative $3$-simplex, with the only crossings coming from $3$-simplices $\tau$ (above) and $\eta$ (below), where only one of $\tau$ or $\eta$ shares an
  edge with $\sigma$. Moreover, there is triple point semiopen line segment $(p_0,p_1]$ in $\sigma$, where $p_0$ lies in the common edge, and $\pi(\tau)\cap \pi(\eta)$ contains $\pi(p_0,p_1]$ in its interior, see Figure ~\ref{F12}.  
  \end{enumerate}
   \begin{figure}[!ht]
  \includegraphics [width=6cm] {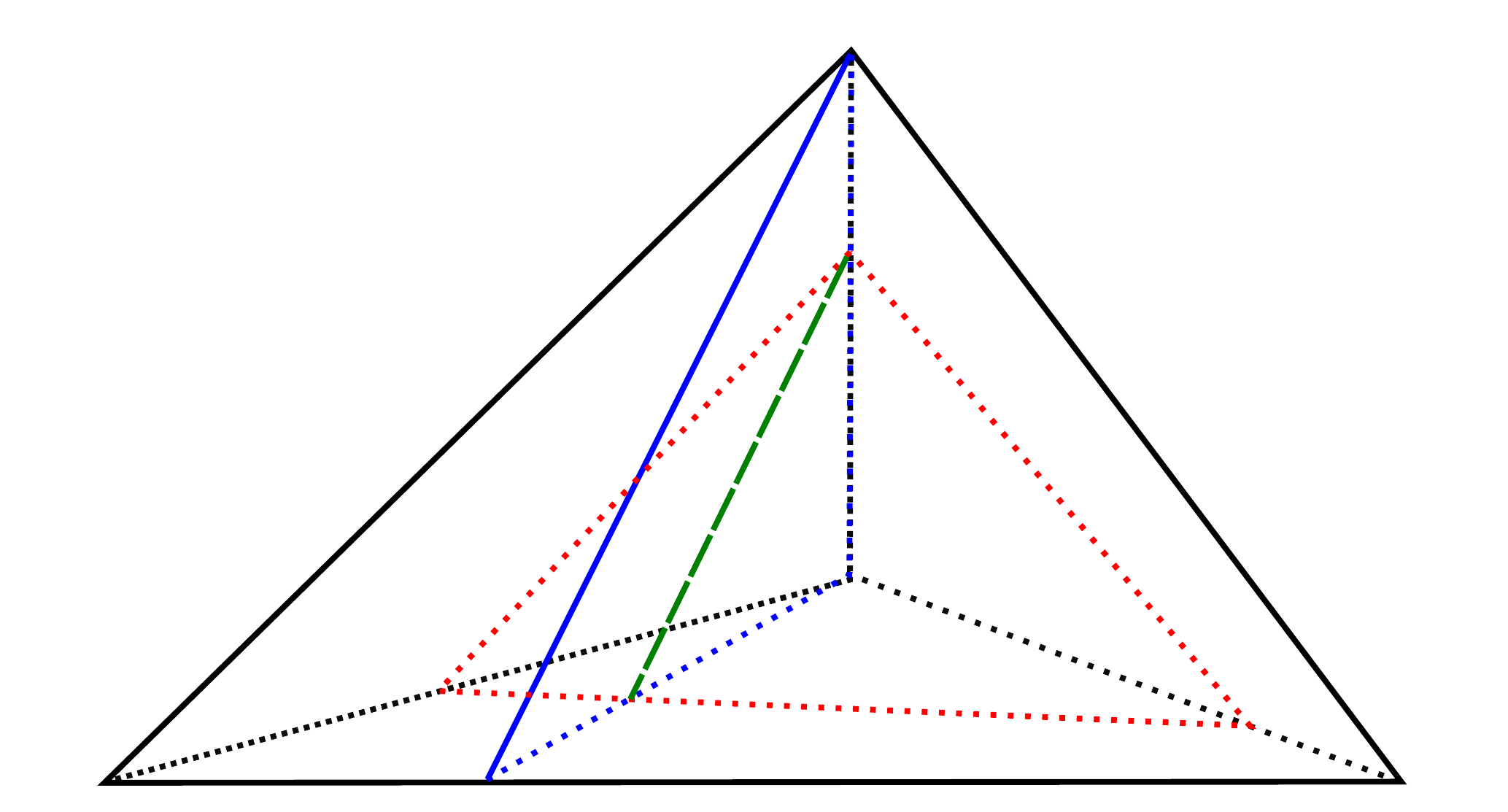} 
  \caption{Preimage Picture: a $3$-simplex  intersecting with a non adjacent $3$-simplex and one sharing an edge in a triple point semiopen line segment, illustrating special case 6.}
  \label{F12}
  \end{figure}
  
 \textit{General Case}. Suppose $\sigma$ is any negative $3$-simplex.
   \begin{figure}[!ht]
  \includegraphics [width=6cm] {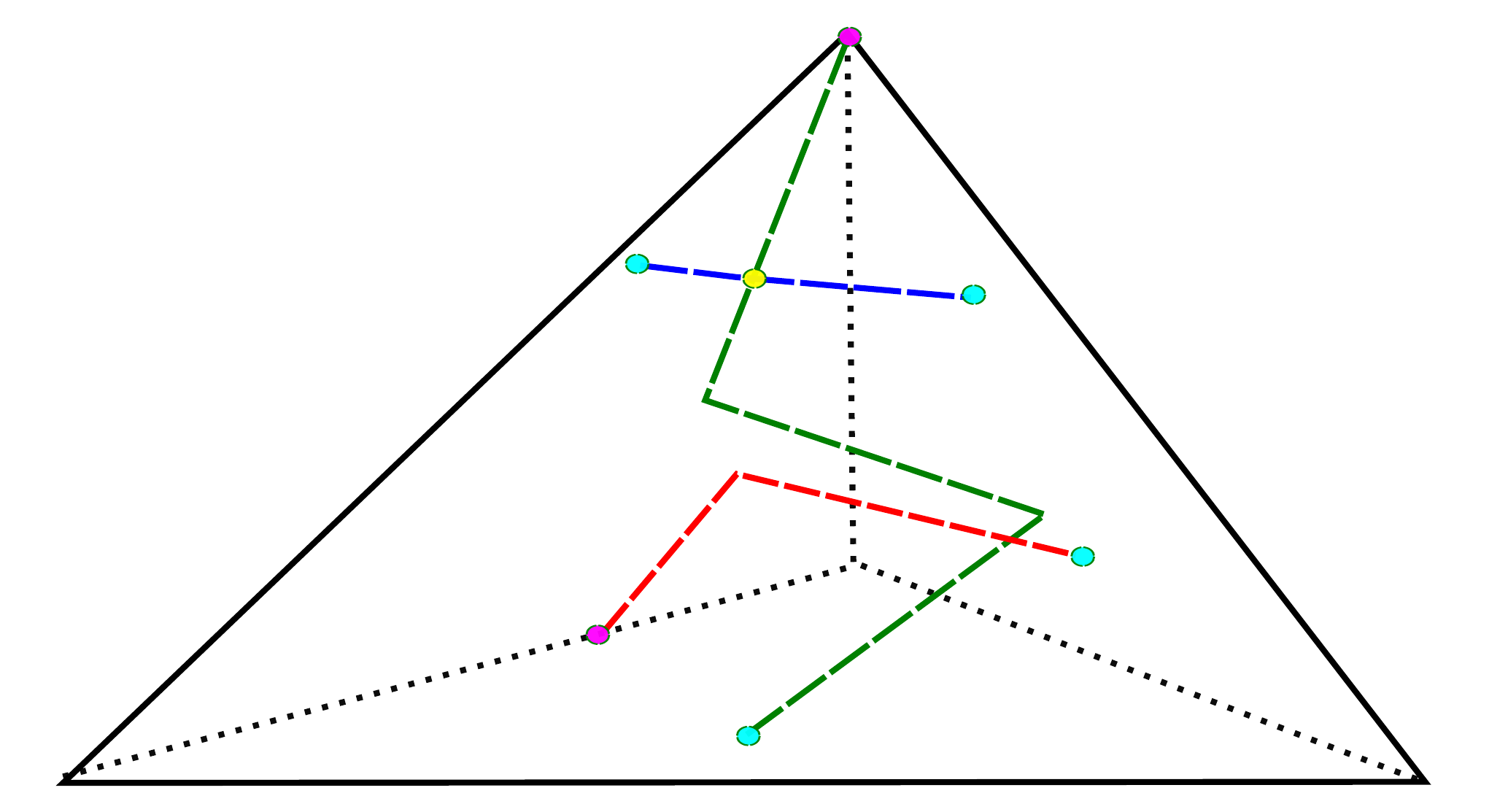} 
  \caption{Preimage Picture: a $3$-simplex with triple points (some of whom converge to a vertex or an edge) and quadruple points, illustrating the general case
  (double points not indicated).}
  \label{F13}
  \end{figure}
  
  Let us first consider all the points where $\mathscr{T}_\sigma$ meets a vertex or an edge of $\sigma$, and by special cases 5 and 6, we can find small inner $3$-simplices
  containing these points where we can apply cellular moves and replace them with positive $3$-simplices.
  We can break the rest $\sigma$ up into small inner $3$-simplices so that we are special case 4,
  and as we have seen, we can replace each of the subsimplices with positive $3$-simplices, thereby reducing the number of negative $3$-simplices by one.
  This completes the proof of Theorem~\ref{1a} for the case $k=3$.
\end{proof}

\subsection{Questions}

For $k>3$, if we have an embedded closed oriented $k$-link in $\mathbb{R}^{k+2}$, and we can make sure (at every step of applying cellular moves) that the triple points of $M$ do not intersect $\delta M$
then we can use the above method to isotope the link to be positive. However, we cannot always guarantee such a condition on the triple point set,
and our approach seems to fail if $k>3$ (i.e.\@ ambient dimension $n>5$).
\begin{question}
 Can Theorem~\ref{1} be extended for higher ambient dimension? If not, can a counterexample/obstrustion be found?
\end{question}
 We can restrict our attention to closed braids where the branched set in $M$ is a $(k-2)$-dimensional submanifold.
In our proof, we could only show that every closed oriented 3-link in $\mathbb{R}^5$ can be isotoped to be a closed braid where the branched set is a graph. 
\begin{question}
 Can every closed oriented 3-link in $\mathbb{R}^5$ be isotoped to be a closed braid where the branched set is a link?
 More generally, for which $n\geq 5$, is every closed oriented $(n-2)$-link in $\mathbb{R}^n$ isotopic to a closed braid with the branched set being a submanifold?
\end{question}
We can also ask similar questions in the smooth category.
\begin{question}
 For which $n$, is every closed oriented smooth $(n-2)$-link in $\mathbb{R}^n$ smoothly isotopic to a closed braid (with the branched set being a submanifold)?
\end{question}


\section{Higher co-dimension braiding}
In the first subsection, we will use the tools developed so far to complete the proof of Theorem~\ref{2}.
We end with some questions about higher co-dimension braidings in the second subsection.
\subsection{Isotoping higher co-dimension link to be positive}
To prove Theorem~\ref{2} it remains to show the following.
\begin{thm}\label{2a}
 Any closed oriented piecewise linear $k$-link  $f:M\rightarrow\mathbb{R}^{k+l}$ can be piecewise linearly isotoped to be a closed braid for $2l\geq k+2$.
\end{thm}

\begin{remark}\label{r1a}
 In case $k=1$ or $2$, then $l=2$ satisfies the hypothesis of the above theorem, and we know in this case the result follows from Theorem~\ref{1a}. 
 In the rest of the section, we will assume that $l\geq 3$. We will also assume that $l\leq k+1$, since otherwise\footnote{The proof given still holds
 if $l>k+1$, one just has to interpret statements (like negative dimensional space) correctly.} it is easy to see that the Theorem holds, as there will be no crossings in the projection under $\pi_v$.
 In fact one can show if $l>k+1$, then any two embeddings of $M^k$ in $\mathbb{R}^{k+l}$ are isotopic (see \cite[Corollary~5.9]{RS}) .
\end{remark}

The proof will be similar to the proof of case $k=2$ of Theorem~\ref{1a}, and we will not discuss special cases of negative simplices this time.
\begin{proof}
 The proof will be by induction on the number of negative simplices in the division of $M$. 
Let us consider the projection under $\pi_v:\mathbb{R}^{k+l}\rightarrow\mathbb{R}^{k+l-1}$, and see what we can say about the crossings under the 
given hypothesis $2l\geq k+2$.

 For any two simplices $\sigma$ and $\tau$, we may assume that $\pi_v(\sigma)$ and $\pi_v(\tau)$ intersect transversely in 
$\pi_v(\mathbb{R}^{k+l})=\mathbb{R}^{k+l-1}$, and in that case the intersection of the affine subspaces generated by them have dimension $2k-(k+l-1)=k-l+1$.
 We may assume that for any triple of simplices $\tau$, $\sigma$ and $\nu$, that $\pi_v(\sigma\cap\tau)$ intersects $\pi_v(\sigma\cap\nu)$ transversely in
 $\pi_v(\sigma)$, and in that case the intersection has dimension $k-2(k-l+1)=k-2l+2\leq 0$. Consequently, all triple points are isolated and can be assumed
 to be in the interior of their respective simplices. 
 
\textit{General Position for the initial triangulation}. We may assume that all the crossings are double point complex is $(k-l+1)$-dimensional CW-complex, all triple points are isolated and in the interior of respective
$k$-simplices, and there are no quadruple points.

\textit{General Position for a division}. When applying a cellular move along $D=-(q*\nu)$, we may assume that
the $q$ and $\nu$ are chosen so that $\pi_v(\overline{\mathcal{D}_{M\setminus{\mathring{\nu}}}})$ meets  $\pi_v(\partial D)$ and $\pi_v(\delta D)$ transversely.
 Consequently, there are no quadruple points, and moreover all triple points are isolated and lie in the interior of their respective $k$-simplices.

 Now given any negative simplex $\sigma$, it will contain finitely many triple points $p_1,...,p_m$ in its interior.
  We can choose points $q_1,....,q_m$ such that $O\in(\pi_v(q_i),\pi_v(p_i))$ and each of these line segments $[q_i,p_i]$ are mutually disjoint, and do not 
  intersect the rest of $M$. Using Remark~\ref{r6} we can find $k$-simplices $P_i$ containing $p_i$ in $\sigma$ such that $q_i*P_i$ are mutually disjoint, and then we can use cellular moves to replace
  each $P_i$ by a union of positive $k$-simplices. The rest of $\sigma$ can be broken up into smaller inner $k$-simplices such that there are only crossings
  of one type, and then by Lemmas~\ref{A} and ~\ref{B}, we can replace $\sigma$ with a union of positive simplices. We have reduced the number of negative
  simplices by one, and hence we are done by induction.
\end{proof}

\subsection{Questions}

 We have shown that in the piecewise linear category that the answer to Question~\ref{1} is affirmative if $n\geq 2k$ for $k\geq 2$, and also if $n=k+2$ if $1\leq k\leq 3$ .
\begin{question} (in both piecewise linear and smooth categories)
 Given a natural number $k$, is there a natural number $n$ such that the answer to Question~\ref{1} is affirmative, and if so the what is the smallest such $n$?
\end{question}
We can also focus on a given manifold and ask:
\begin{question}
 Given a smooth closed oriented $k$-manifold $M$, is there a smooth branched cover over the sphere $S^k$?
\end{question}
We know that the answer is yes to the corresponding question in the piecewise linear category, due to Alexander \cite{A0} (see also Remark~\ref{r5}).

\begin{question} (in both piecewise linear and smooth categories)
 Given a closed oriented $k$-manifold $M$ which is a branched cover over the $k$-sphere, what is the minimum\footnote{In the smooth category, there will always be a braided embedding of $M$ to some $\mathbb{R}^N$ if we know there is
 a smooth branched cover $g:M\rightarrow S^k$. We know there is a smooth embedding $i:M\rightarrow\mathbb{R}^{2k}$. The map $f=g\times i:M\rightarrow N(S^k)\subseteq \mathbb{R}^{3k}$ is a braided embedding.} $n$ such that there is a braided embedding of $M$ in $\mathbb{R}^{n}$. Is there an $M$ such that this $n$ is
 larger than the smallest dimensional Euclidean space into which $M$ embeds?
\end{question}
It is very likely that in some cases the condition $2l\geq k+2$ we had in Theorem~\ref{2a} can be weakened and still any embedding can be braided.
\begin{question} (in both piecewise linear and smooth categories)
 Given a closed oriented $k$-manifold $M$, what is the range for $l\geq 2$ such that $M$ embeds in $\mathbb{R}^{k+l}$, and every embedding is isotopic to a closed braid?
 If $l$ is not in that range, can we find a counterexample/obstrustion?
\end{question}



\begin{thebibliography}{11}
 \bibitem{A0}
 J.W. Alexander.
 \newblock Note on Riemann spaces.
 \newblock {\em Bull. Amer. Math. Soc. }, 26:370–372, 1920.

 
 \bibitem{A}
 J.W. Alexander.
 \newblock A lemma on systems of knotted curves.
 \newblock {\em Proc. Nat. Acad. Sci. USA}, 9:93--95, 1923.

 \bibitem{B} J.S. Birman.
 \newblock Braids, Links, and Mapping Class Groups.
 \newblock {\em Annals of Mathematics Studies}, Princeton University Press, 1974.
 
  \bibitem{CK} J.S. Carter and S. Kamada. 
 \newblock How to Fold a Manifold.
 \newblock {\em ArXiv e-prints}, January 2013.

 \bibitem{EF} J.B. Etnyre and R. Furukawa.
 Braided Embeddings of Contact 3–manifolds
 in the standard Contact 5–sphere,
 \newblock {\em ArXiv e-prints}, October 2015.
 
 \bibitem{HLM} H.M. Hilden, M.T. Lozano and J.M. Montesinos.
 All three-manifolds are pull-backs of a branched covering $S^3$ to $S^3$,
 \newblock {\em Trans. Amer. Math. Soc.}, 279(2):729–735, 1983.
 
 \bibitem{J} V.F.R. Jones.
 \newblock Hecke algebra representation of braid groups and link polynomials.
 \newblock {\em Annals of Mathematics}, 126:335-388, 1987.
 
 \bibitem{K1} S. Kamada.
 \newblock  A characterization of groups of closed orientable surfaces 4-space.
 \newblock {\em Topology}, 33:113--122, 1994.

 \bibitem{K2} S. Kamada.
 \newblock  Braid and Knot Theory in Dimension Four.
 \newblock {\em American Mathematical Society}, 2002.
 
 \bibitem{RS} C.P. Rourke and B.J. Sanderson.
 \newblock  Introduction to Piecewise Linear Topology.
 \newblock {\em Springer-Verlag}, 1972.
 
  \bibitem{R} L. Rudolph.
 Braided surfaces and Seifert ribbons for closed braids.
 \newblock {\em Comment. Math. Helv.}, 58(1):1–37, 1983.

 \bibitem{W} H. Whitney.
 The Self-Intersections of a Smooth $n$-Manifold in $2n$-Space.
 \newblock {\em Annals of Mathematics}, 45(2):220–246, 1944.
 
 \end{thebibliography}
\end{document}